\documentclass[letterpaper]{amsart}
\usepackage{amssymb,xspace,cmap}

\newtheorem*{thma}{Theorem~A}
\newtheorem*{cora}{Corollary~A}
\newtheorem*{corb}{Corollary~B}
\newtheorem{thm}{Theorem}[section]
\newtheorem{cor}[thm]{Corollary}
\newtheorem{fact}[thm]{Fact}
\newtheorem{lemma}[thm]{Lemma}
\newtheorem{claim}{Claim}[thm]

\theoremstyle{definition}
\newtheorem{defn}[thm]{Definition}

\theoremstyle{remark}
\newtheorem{remark}{Remark}

\newcommand*\axiomfont[1]{\textsf{\textup{#1}}\xspace}
\newcommand\ch{\axiomfont{CH}}
\newcommand\MA{\axiomfont{MA}_{\omega_1}}
\newcommand\MM{\axiomfont{MM}}
\newcommand\PFA{\axiomfont{PFA}}
\newcommand\ZFC{\axiomfont{ZFC}}
\newcommand\ZFCm{\axiomfont{ZFC}^-}
\newcommand\SCFA{\axiomfont{SCFA}}
\newcommand{\V}{\ensuremath{\mathrm{V}}}
\newcommand{\seq}[2]{{\langle#1\;|\;}\linebreak[0]{#2\rangle}} 
\newcommand{\kla}[1]{ {\langle #1 \rangle} }

\DeclareMathOperator{\cf}{cf}
\DeclareMathOperator{\tr}{Tr}
\DeclareMathOperator{\ran}{Im}
\DeclareMathOperator{\otp}{otp}
\DeclareMathOperator{\dom}{dom}
\DeclareMathOperator{\add}{Add}
\DeclareMathOperator{\acc}{acc}
\DeclareMathOperator{\refl}{Refl}

\author{Gunter Fuchs}
\address{The College of Staten Island (CUNY), 2800 Victory Blvd., Staten Island, NY 10314\\
The Graduate Center (CUNY), 365 5th Avenue, New York, NY 10016}
\urladdr{http://www.math.csi.cuny.edu/\textasciitilde fuchs}

\author{Assaf Rinot}
\address{Department of Mathematics, Bar-Ilan University, Ramat-Gan 5290002, Israel.}
\urladdr{http://www.assafrinot.com}

\thanks{The first author was partially supported by PSC-CUNY grant 69656-00 47.
The second author was partially supported by the Israel Science Foundation (grant $\#$1630/14).\\
\indent The final version of this article is to appear in \emph{Acta Mathematica Hungarica.}}

\begin{document}
\title{Weak square and stationary reflection}
\begin{abstract}It is well-known that the square principle $\square_\lambda$ entails the existence of a non-reflecting stationary subset of $\lambda^+$,
whereas the weak square principle $\square^*_\lambda$ does not.
Here we show that if $\mu^{\cf(\lambda)}<\lambda$ for all $\mu<\lambda$,
then $\square^*_\lambda$ entails the existence of a non-reflecting stationary subset of $E^{\lambda^+}_{\cf(\lambda)}$
in the forcing extension for adding a single Cohen subset of $\lambda^+$.

It follows that indestructible forms of simultaneous stationary reflection entail the failure of weak square.
We demonstrate this by settling a question concerning the subcomplete forcing axiom (\SCFA),
proving that \SCFA entails the failure of $\square^*_\lambda$ for every singular cardinal $\lambda$ of countable cofinality.
\end{abstract}
\date{\today}

\maketitle

\section{Introduction}
Write $E^\mu_\chi:=\{\alpha<\mu\mid \cf(\alpha)=\chi\}$, and likewise $E^\mu_{>\chi}:=\{\alpha<\mu\mid \cf(\alpha)>\chi\}$.
\begin{defn}[Trace] For a regular uncountable cardinal $\mu$ and a stationary subset $S\subseteq\mu$,
let $\tr(S):=\{\alpha\in E^\mu_{>\omega}\mid S\cap\alpha\text{ is stationary in }\alpha\}$.
\end{defn}

If $\tr(S)$ is nonempty, then we say that $S$ \emph{reflects}. Otherwise, we say that $S$ is \emph{non-reflecting}.

\begin{defn}[Simultaneous reflection]
The principle $\refl(\theta,S)$ asserts that for every sequence $\langle S_i \mid i<\theta\rangle$ of stationary subsets of $S$,
the set $\bigcap_{i<\theta}\tr(S_i)$ is nonempty.
\end{defn}

\begin{defn}[Square] For an infinite cardinal $\lambda$ and a nonzero cardinal $\mu\le\lambda$,
 $\square_{\lambda,\mu}$ asserts the existence of a sequence $\langle C_\delta\mid\delta<\lambda^+\rangle$ such that:
\begin{enumerate}
\item for all $\delta\in\acc(\lambda^+)$, $C_\delta$ is a club in $\delta$ of order-type $\le\lambda$;
\item $|\{ C_\delta\cap\gamma\mid \delta<\lambda^+\ \&\ \sup(C_\delta\cap\gamma)=\gamma\}|\le\mu$ for all $\gamma<\lambda^+$.
\end{enumerate}
\end{defn}
\begin{remark} Here, $\acc(x):=\{\delta\in x\mid \sup(x\cap\delta)=\delta>0\}$.
\end{remark}

The strongest principle $\square_{\lambda,1}$ is denoted by $\square_\lambda$,
and the weakest principle $\square_{\lambda,\lambda}$ is denoted by $\square^*_\lambda$.
By Fodor's lemma, $\square_\lambda$ entails the failure of $\refl(S,1)$ for every stationary $S\subseteq\lambda^+$.
More generally:
\begin{fact}[Cummings and Magidor, {\cite[Lemma 2.1, 2.2]{MR2811288}}]\label{squareNreflection} For all infinite cardinals $\mu<\lambda$:
\begin{enumerate}
\item $\square_{\lambda,\mu}$ entails the failure of $\refl(\cf(\lambda),S)$ for every stationary $S\subseteq\lambda^+$;
\item If $\mu<\cf(\lambda)$, then $\square_{\lambda,\mu}$ entails the failure of $\refl(1,S)$ for every stationary $S\subseteq\lambda^+$.
\end{enumerate}
\end{fact}
In contrast, by Theorem~12.1 of \cite{MR1838355}, $\square^*_{\aleph_\omega}$ is compatible with $\refl(\theta,S)$ holding for every $\theta<\aleph_\omega$ and every stationary $S\subseteq\aleph_{\omega+1}$.
The purpose of this note is to show that while $\square^*_{\lambda}$ does not imply the failure of $\refl(1,E^{\lambda^+}_{\cf(\lambda)})$,
it does imply it in some forcing extension. To be more precise:

\begin{thma} Suppose $\lambda$ is a singular cardinal,
and $\mu^{\cf(\lambda)}<\lambda$ for all $\mu<\lambda$.

If $\square^*_\lambda$ holds, then in the generic extension by $\add(\lambda^+,1)$, there exists a non-reflecting stationary subset of $E^{\lambda^+}_{\cf(\lambda)}$.
\end{thma}
\begin{remark}
Here, $\add(\kappa,\theta)$ stands for Cohen's notion of forcing for adding $\theta$ many subsets of $\kappa$.
Specifically, conditions are functions $f:a\rightarrow2$,
with $a\subseteq\kappa\times\theta$ and $|a|<\kappa$. A condition $f$ extends $g$ iff $f\supseteq g$.
\end{remark}

Theorem~A implies that various forcing axioms and large cardinal axioms that may be preserved by Cohen forcing entail the failure of weak square.
Proofs in this vein will go through the following consequence of Theorem~A.

\begin{cora} Suppose $\theta<\kappa$ are regular cardinals,
and there exists a notion of forcing $\mathbb P$ such that:
\begin{itemize}
\item[(1)] $\mathbb P$ does not change cofinalities,
\item[(2)] $\V^{\mathbb P}\models \kappa^\theta=\kappa$,
\end{itemize}
and for every $\dot{\mathbb{Q}}$ such that $\mathbb P$ forces that $\dot{\mathbb{Q}}$ is a ${<}\kappa$-directed closed notion of forcing:
\begin{itemize}
\item[(3)] $\V^{\mathbb P*\dot{\mathbb{Q}}}\models \refl(\theta,E^{\mu}_{\theta})$ for every regular cardinal $\mu>\kappa$.
\end{itemize}
Then $\square^*_\lambda$ fails for every cardinal $\lambda>\kappa$ of cofinality $\theta$.
\end{cora}

In this paper, we provide an application to the subcomplete forcing axiom (\SCFA),
answering Question~2.12 of \cite{Fuchs:HierarchiesOfForcingAxioms} in the affirmative, by proving:
\begin{corb} \SCFA entails that $\square^*_\lambda$ fails for every singular cardinal $\lambda$ of countable cofinality.
\end{corb}

\section{Weak square vs. reflection}\label{sec:WeakSquareVsReflection}

\begin{defn}Suppose $\lambda$ is an uncountable cardinal and $S\subseteq\lambda^+$ is stationary.
\begin{enumerate}
\item $\diamondsuit^*(S)$ asserts the existence of a sequence $\langle \mathcal A_\delta\mid \delta\in S\rangle$ such that:
\begin{enumerate}
\item for every $\delta\in S$, $\mathcal A_\delta\subseteq\mathcal P(\delta)$ and $|\mathcal A_\delta|\le \lambda$;
\item for every subset $Z$ of $\lambda^+$, the following set is nonstationary:
$$\left\{\delta\in S\mid \forall A\in\mathcal A_\delta( A\not= Z\cap \delta)\right\}.$$
\end{enumerate}
\item $\clubsuit^*(S)$ asserts the existence of a sequence $\langle \mathcal A_\delta\mid \delta\in S\rangle$ such that:
\begin{enumerate}
\item for every $\delta\in S$, $\mathcal A_\delta\subseteq[\delta]^{<\lambda}$ and $|\mathcal A_\delta|\le\lambda$;
\item for every cofinal subset $Z$ of $\lambda^+$, the following set is nonstationary:
$$\left\{\delta\in S\mid \forall A\in\mathcal A_\delta( \sup(A\cap Z)<\delta)\right\}.$$
\end{enumerate}
\end{enumerate}
\end{defn}

\begin{fact}\label{Weaksquare} Suppose $\lambda$ is a singular cardinal, and $\square^*_\lambda$ holds.

If every stationary subset of $E^{\lambda^+}_{\cf(\lambda)}$ reflects, then $\clubsuit^*(\lambda^+)$ holds.
\end{fact}
\begin{proof} By Corollaries 3.12 and 2.15 of \cite{MR2723781}.
\end{proof}

Let $\ch_\lambda$ stand for the assertion that $2^\lambda=\lambda^+$.
The following is a slight refinement of Theorem~1.6 of \cite{MR2723781}. We use the notation $\cf(X,{\supseteq})$ for the least $\kappa$ such that there is a set $D\subseteq X$ of cardinality $\kappa$ that is cofinal in $(X,{\supseteq})$, meaning that for every $a\in X$, there is a $b\in D$ with $b\subseteq a$.

\begin{lemma}\label{15} Suppose $\theta<\lambda$ are infinite cardinals, $S\subseteq E^{\lambda^+}_\theta$ is stationary,
and $\cf([\mu]^\theta,{\supseteq})\le\lambda$ for all $\mu\in[\theta,\lambda)$.
Then $\clubsuit^*(S)+\ch_\lambda$ iff $\diamondsuit^*(S)$.
\end{lemma}
\begin{proof} We focus on the forward implication.
As $\cf([\mu]^\theta,{\supseteq})\le\lambda$ for all $\mu\in[\theta,\lambda)$,
for each $A\in\mathcal P(\lambda^+)$ with $\theta\le|A|<\lambda$,
let $\mathcal D_A$ be cofinal in the poset $([A]^\theta,{\supseteq})$, with $|\mathcal D_A|\le\lambda$.
As $\ch_\lambda$ holds, let $\langle X_\beta\mid \beta<\lambda^+\rangle$ be an enumeration of $[\lambda^+]^{\le\lambda}$
such that every element is enumerated cofinally often.
For every set $B\subseteq\lambda^+$, denote
$$(B)_*:=\bigcup\{X_\beta\mid \beta\in B\}.$$

Let $\langle \mathcal A_\delta\mid \delta\in S\rangle$ be a $\clubsuit^*(S)$-sequence.
For each $\delta\in S$, let:
$$\mathcal B_\delta:=\{ (B)_*\mid \exists A\in\mathcal A_\delta(B\in\mathcal D_A)\}\cap\mathcal P(\delta).$$

Clearly, $|\mathcal B_\delta|\le\lambda$. Thus, to see that $\langle \mathcal B_\delta\mid\delta\in S\rangle$ is a $\diamondsuit^*(S)$-sequence,
let $X$ be an arbitrary subset of $\lambda^+$.
We shall find a club $C\subseteq\lambda^+$ such that $X\cap\delta\in\mathcal B_\delta$ for all $\delta\in S\cap C$.

Recalling that for every $x\in[\lambda^+]^{\le\lambda}$, $\{\beta<\lambda^+\mid x=X_\beta\}$ is cofinal in $\lambda^+$,
we may fix a strictly increasing function $f:\lambda^+\rightarrow\lambda^+$ such that for all $\alpha<\lambda^+$:
$$X_{f(\alpha)}=X\cap\alpha.$$

As $Y:=\ran(f)$ is a cofinal subset of $\lambda^+$, and $D:=\{\delta<\lambda^+\mid f[\delta]\subseteq\delta\}$ is a club in $\lambda^+$,
we may find some sparse enough cofinal subset $Z$ of $Y$ such that for any pair $\beta<\beta'$ of elements from $Z$,
there exists some $\delta\in D$ with $\beta<\delta<\beta'$.

Next, fix a club $C\subseteq\lambda^+$
such that for all $\delta\in S\cap C$, there exists some $A\in\mathcal A_\delta$ with $\sup(A\cap Z)=\delta$.

Let $\delta\in S\cap C$ be arbitrary.
Fix some $A\in\mathcal A_\delta$ such that $\sup(A\cap Z)=\delta$.
By $\cf(\delta)=\theta$, let us fix a cofinal subset $A'$ of $A\cap Z$ of order-type $\theta$.
As $A'\in[A]^\theta$, and $\mathcal D_A$ is cofinal in $([A]^\theta,{\supseteq})$,
let us fix some $B\in\mathcal D_A$ with $B\subseteq A'$. As $|B|=\theta=\otp(A')$,
we infer that $\sup(B)=\sup(A')=\delta$.
As $B\subseteq A'\subseteq A\subseteq Z\subseteq Y=\ran(f)$, we infer the existence of some $a\subseteq\lambda^+$ such that $f\restriction a$ is an order-preserving bijection from $a$ to $B$.
By $B\subseteq Z$ and the choice of $Z$, for any pair of ordinals $\alpha<\alpha'$ from $a$, there exists some $\delta\in D$ such that $f(\alpha)<\delta<f(\alpha')$.
As $\delta\in D$ implies $f[\delta]\subseteq\delta$, this means that $\alpha\le\delta\le\alpha'$.
Consequently, $\sup(a)=\sup(B)=\delta$,
so that
\begin{equation*}
(B)_*=\bigcup\{X_\beta\mid \beta\in B\}=\bigcup\{ X_{f(\alpha)}\mid \alpha\in a\}=X\cap\sup(a)=X\cap\delta.
\end{equation*}

So $X\cap \beta=(B)_*\in\mathcal B_\delta$, and we are done.
\end{proof}

Recall that $\log_\lambda(\lambda^+)$ stands for the least cardinal $\theta$ such that $\lambda^\theta>\lambda$.
The following is a refinement of Theorem~3.2 of \cite{MR523488}, which itself builds on an idea from \cite[p.~387]{MR0277379}.
\begin{lemma}\label{l14} Suppose $\lambda$ is an infinite cardinal, and $\vec{\mathcal A}=\langle\mathcal A_\delta\mid\delta\in E^{\lambda^+}_{\theta}\rangle$
is a $\diamondsuit^*(E^{\lambda^+}_{\theta})$-sequence, for $\theta:=\log_\lambda(\lambda^+)$.
Then in $\V^{\add(\lambda^+,1)}$, $\vec{\mathcal A}$ is no longer a $\diamondsuit^*(E^{\lambda^+}_{\theta})$-sequence.
\end{lemma}
\begin{proof}
For simplicity, we identify $\add(\lambda^+,1)$ with the collection ${}^{<\lambda^+}\lambda$, ordered by inclusion.

In $\V$, fix a bijection $\pi:\lambda^+\leftrightarrow\lambda^+\times\lambda$, and let $D:=\{\delta<\lambda^+\mid \pi[\delta]=\delta\times\lambda\}$,
so that $D$ is a club in $\lambda^+$.
Towards a contradiction, suppose that $G$ is $\add(\lambda^+,1)$-generic over $\V$,
and $\V[G]\models \vec{\mathcal A}\text{ is a }\diamondsuit^*(E^{\lambda^+}_{\theta})\text{-sequence.}$

Work in $\V[G]$. Let $g:=\bigcup G$ be the generic function from $\lambda^+$ to $\lambda$,
and then put $Z:=\pi^{-1}[g]$.
Fix a club $C\subseteq\lambda^+$ such that for all $\delta\in E^{\lambda^+}_{\theta}\cap C$,
we have $Z\cap\delta\in\mathcal A_\delta$. Without loss of generality, we may assume that $C=C\cap D$.

Work back in $\V$. Fix a condition $p\in G$, a nice name $\dot C$ for $C$,
and a canonical name $\dot Z$ for $Z$,
such that $p$ forces: $\dot{C}\subseteq D$, and $\dot{Z}\cap\delta\in\mathcal A_\delta$ for all $\delta\in E^{\lambda^+}_{\theta}\cap\dot{C}$.

We shall now define a sequence $\langle (\beta_\ell,\langle (p_s,\delta_s)\mid s\in{}^{\ell}\lambda\rangle)\mid \ell\le\theta\rangle$
in such a way that:
\begin{enumerate}
\item For any $s\in{}^{\le\theta}\lambda$, $p_s$ is a condition extending $p$, forcing that $\delta_s$ is in $C$;
\item For all $s\subseteq t$ from ${}^{\le\theta}\lambda$, we have $p_s\subseteq p_t$;
\item For all $\ell<\theta$, $s\in{}^{\ell}\theta$, and $i<\lambda$:
\begin{enumerate}
\item $\beta_\ell\le\delta_s<\beta_{\ell+1}<\delta_{s{}^\smallfrown\langle i\rangle}$ and $\delta_{s{}^\smallfrown\langle i\rangle}\in D$;
\item $\beta_{\ell+1}\in \dom(p_{s{}^\smallfrown\langle i\rangle})$ and $p_{s{}^\smallfrown\langle i\rangle}(\beta_{\ell+1})=i$.
\end{enumerate}
\end{enumerate}

The definition is by recursion on $\ell\le\theta$:

$\blacktriangleright$ Let $\beta_0:=0$, and find some condition $p_\emptyset\supseteq p$ along with an ordinal $\delta_\emptyset$,
such that $p_{\emptyset}\Vdash\check{\delta_\emptyset}\in\dot C$.

$\blacktriangleright$ Suppose that $\ell<\theta$,
and that $\langle (\beta_l,\langle (p_s,\delta_s)\mid s\in{}^{l}\lambda\rangle)\mid l\le \ell\rangle$ has already been defined.
Put $\beta_{\ell+1}:=\sup\{ \dom(p_s), \delta_s\mid s\in{}^{\ell}\lambda\}+1$.
By $\ell<\theta$, we have $|{}^\ell\lambda|<\lambda^+$, so that $\beta_{\ell+1}<\lambda^+$.
For all $s\in{}^{\ell}\lambda$ and $i<\lambda$,
define $p_s^i:\beta_{\ell+1}+1\rightarrow\lambda$ by stipulating:
$$p_s^i(\alpha):=\begin{cases}p_s(\alpha),&\text{if }\alpha\in\dom(p_s);\\
i,&\text{otherwise}.
\end{cases}$$
Next, since $\dot C$ is a name for an unbounded subset of $\lambda^+$, find an ordinal $\delta_{s{}^\smallfrown\langle i\rangle}>\beta_{\ell+1}$ and a condition $p_{s{}^\smallfrown\langle i\rangle}$ extending $p_s^i$ such that
$p_{s{}^\smallfrown\langle i\rangle}\Vdash\check{\delta_s}\in\dot C$.

$\blacktriangleright$ Suppose that $\ell\in\acc(\theta+1)$,
and that $\langle (\beta_l,\langle (p_s,\delta_s)\mid s\in{}^{l}\lambda\rangle)\mid l<\ell\rangle$ has already been defined.
Put $\beta_\ell:=\sup_{l<\ell}\beta_l$.
For every $s\in{}^{\ell}\lambda$, let $p_s:=\bigcup\{p_{s\restriction l}\mid l<\ell\}$ and $\delta_s:=\sup\{\delta_{s\restriction l}\mid l<\ell\}$.
By $\ell<\theta^+\le\lambda^+$, $p_s$ is a condition.
Since $\dot C$ is a name for a closed subset of $\lambda^+$, we have $p_s\Vdash \check \delta_s\in\dot C$.
Also, it is clear that $\delta_ s=\beta_\ell$.

This completes the construction.

\medskip

Put $\delta:=\beta_\theta$, so that $\delta\in E^{\lambda^+}_\theta\cap D$.
For each $s\in{}^\theta\lambda$, put $Z_s:=\pi^{-1}[p_s]$.
By $\dom(p_s)=\delta\in D$, we have $p_s\Vdash \dot Z\cap \check\delta=\check{Z_s}$ and $p_s\Vdash\check\delta\in\dot C$.
By Clause~(3)(b) above, for any distinct $s,t\in{}^\theta\lambda$, we have $p_s\neq p_t$ and $Z_s\neq Z_t$,
so since $|\mathcal A_\delta|<\lambda^+\le\lambda^\theta$, we may find some $s\in{}^\theta\lambda$ such that $Z_s\notin\mathcal A_\delta$.
However, this contradicts the fact that $p_s$ extends $p$.
\end{proof}

We are now ready to prove the theorem mentioned in the Introduction.

\begin{thma}
Suppose $\lambda$ is a singular cardinal,
and $\mu^{\cf(\lambda)}<\lambda$ for all $\mu<\lambda$.

If $\square^*_\lambda$ holds, then in the generic extension by $\add(\lambda^+,1)$, there exists a non-reflecting stationary subset of $E^{\lambda^+}_{\cf(\lambda)}$.
\end{thma}
\begin{proof} Suppose that $\square^*_\lambda$ holds. Let $G$ be ${\add(\lambda^+,\lambda^{++})}$-generic over $\V$.
For all $\alpha\le\lambda^{++}$, let $G_\alpha$ denote the induced generic for ${\add(\lambda^+,\alpha)}$.

\begin{claim} All of the following hold true in $\V[G]$:
\begin{enumerate}
\item $\cf([\mu]^{\cf(\lambda)},{\supseteq})<\lambda$ for all $\mu\in[\cf(\lambda),\lambda)$;
\item $\log_\lambda(\lambda^+)=\cf(\lambda)$;
\item $\ch_\lambda$;
\item $\neg\diamondsuit^*(E^{\lambda^+}_{\cf(\lambda)})$.
\end{enumerate}
\end{claim}
\begin{proof}
(1) As $\add(\lambda^+,\lambda^{++})$ does not add new subsets of $\lambda$,
we have in $\V[G]$ that $\cf([\mu]^{\cf(\lambda)},{\supseteq})\le\mu^{\cf(\lambda)}<\lambda$ for all $\mu\in[\cf(\lambda),\lambda)$.

(2) As $\add(\lambda^+,\lambda^{++})$ does not add new subsets of $\lambda$,
we have in $\V[G]$ that $$\lambda^{<\cf(\lambda)}=\sum_{\mu<\lambda}\mu^{<\cf(\lambda)}\le\sum_{\mu<\lambda}\mu^{\cf(\lambda)}=\lambda,$$
so that $\log_\lambda(\lambda^+)\ge\cf(\lambda)$. By K\"onig's lemma, $\log_\lambda(\lambda^+)\le\cf(\lambda)$.

(3) By Exercise~G4 of \cite[\S~VII]{MR756630}, $\V[G_\alpha]\models \ch_\lambda$ for all nonzero $\alpha\le\lambda^{++}$.

(4) Suppose not.
Fix a $\diamondsuit^*(E^{\lambda^+}_{\cf(\lambda)})$-sequence, $\vec{\mathcal A}=\langle \mathcal A_\delta\mid \delta\in E^{\lambda^+}_{\cf(\lambda)}\rangle$ in $\V[G]$.
As $\vec{\mathcal A}$ has hereditary cardinality less than $\lambda^{++}$, there exists an infinite $\alpha<\lambda^{++}$ such that $\vec{\mathcal A}\in\V[G_\alpha]$.
However, By Lemma~\ref{l14}, $\vec{\mathcal A}$ is no longer a $\diamondsuit^*(E^{\lambda^+}_{\cf(\lambda)})$-sequence in $V[G_{\alpha+1}]$, let alone in $\V[G]$.
This is a contradiction.
\end{proof}
It now follows from Lemma~\ref{15} that $\V[G]\models\neg\clubsuit^*(E^{\lambda^+}_{\cf(\lambda)})$.
In particular, $\V[G]\models\neg\clubsuit^*(\lambda^+)$.
As $\V$ and $\V[G]$ have the same cardinal structure up to $\lambda^+$ (including), it follows that
$\V[G]\models \square^*_\lambda$, and
we infer from Fact~\ref{Weaksquare} that in $\V[G]$, there exists a non-reflecting stationary subset $S$ of $E^{\lambda^+}_{\cf(\lambda)}$.
As $S$ has hereditary cardinality less than $\lambda^{++}$, there exists some infinite $\alpha<\lambda^{++}$ such that $S\in\V[G_\alpha]$.
Since $|\alpha|\le\lambda^+$, we get that $\add(\lambda^+,\alpha)$ is isomorphic to $\add(\lambda^+,1)$.
Since $\add(\lambda^+,1)$ is homogeneous, this shows that, over our ground model, $\add(\lambda^+,1)$
adds a non-reflecting stationary subset of $E^{\lambda+}_{\cf(\lambda)}$.
\end{proof}

\begin{remark} The preceding improves Theorem~3.15 of \cite{MR2723781}.
\end{remark}

It is easy to see that Corollary~A is a consequence of the following.

\begin{cor}\label{cor26}
Suppose $\lambda>\kappa>\theta=\cf(\lambda)$ are infinite cardinals,
and there exists a notion of forcing $\mathbb P$ such that:
\begin{enumerate}
\item $\mathbb P$ does not change cofinalities;
\item $\V^{\mathbb P}\models \kappa^\theta=\kappa$;
\item $\V^{\mathbb P*\add(\lambda^+,1)}\models\refl(\theta,E^\mu_\theta)$ for every regular cardinal $\mu$ with $\lambda>\mu>\kappa$.
\end{enumerate}

Then $\square^*_\lambda$ fails.
\end{cor}

\begin{proof} Towards a contradiction, suppose that $\square^*_\lambda$ holds.

Work in $\V^{\mathbb P}$. By Clause~(1), $\square^*_\lambda$ holds.
\begin{claim} $\mu^{\theta}<\lambda$ for all $\mu<\lambda$.
\end{claim}
\begin{proof} Suppose not. As $\lambda$ is a limit cardinal, let $\mu$ be the least regular cardinal $<\lambda$ to satisfy $\mu^{\theta}\ge\lambda$.
By K\"onig's lemma, then, $\mu^{\theta}>\lambda$.
By Clause~(2), $\theta<\kappa<\mu<\lambda$.

For every limit ordinal $\alpha<\mu$,
fix a strictly increasing and continuous function $c_\alpha:\cf(\alpha)\rightarrow\alpha$ whose range is cofinal in $\alpha$.
Let $\langle S_i\mid i<\mu\rangle$ be some partition of $E^\mu_\theta$ into pairwise disjoint sets.
For every $B\subseteq\mu$, let $(B)_*:=\{ i<\mu\mid B\cap S_i\neq\emptyset\}$.
Put $$\mathcal P:=\{ (c_\alpha[A])_*\mid \alpha\in E^\mu_{>\omega}, A\in[\cf(\alpha)]^\theta\}.$$

By minimality of $\mu$, for every regular uncountable cardinal $\chi<\mu$, we have $\chi^\theta<\lambda$.
Consequently, $|\mathcal P|\le\lambda$. Thus, to meet a contradiction, it suffices to show that $[\mu]^\theta\subseteq\mathcal P$.

To this end, let $I\in[\mu]^\theta$ be arbitrary.
As $\add(\lambda^+,1)$ adds no new subsets of $\mu$, we infer from Clause~(3) the existence
of an ordinal $\alpha\in\bigcap_{i\in I}\tr(S_i)$.
As $\ran(c_\alpha)$ is a club in $\alpha$, we have $\ran(c_\alpha)\cap S_i\neq\emptyset$ for all $i\in I$.
For all $i\in I$, pick $\alpha_i\in\ran(c_\alpha)\cap S_i$, and let $A:=\{c_\alpha^{-1}(\alpha_i)\mid i\in I\}$.
Evidently, $A\in[\cf(\alpha)]^\theta$, so that $I=(c_\alpha[A])_*\in\mathcal{P}$.
\end{proof}

It now follows from Theorem~A that $\V^{\mathbb P*\add(\lambda^+,1)}\models\neg\refl(1,E^{\lambda^+}_{\cf(\lambda)})$, contradicting Clause~(3).
\end{proof}

\section{An application: \SCFA and weak square}

Recall that the forcing axiom for a class $\Gamma$ of forcing notions says that whenever $\mathbb P\in\Gamma$ and $\Delta$ is a collection of dense subsets of $\mathbb P$ with $|\Delta|\le\omega_1$, there is a filter over $\mathbb P$ that is $\Delta$-generic, that is, that meets every set in $\Delta$. The most well-known forcing axioms are Martin's axiom $\MA$ (the forcing axiom for the class of c.c.c.~forcing), the proper forcing axiom \PFA (the forcing axiom for the class of proper forcing notions), and Martin's Maximum \MM (the forcing axiom for the class of forcing notions which preserve stationary subsets of $\omega_1$). The first author has recently studied forcing axioms and other forcing principles for the class of subcomplete forcing notions, a forcing class that was introduced by Jensen in \cite{Jensen:SPSCF},
and we will use the methods of Section~\ref{sec:WeakSquareVsReflection} to determine precisely the effects of the subcomplete forcing axiom, \SCFA, on the extent of weak square principles. A similar project for \MM was carried out by Cummings and Magidor in \cite{MR2811288}, where the authors proved the following result:

\begin{fact}[Cummings and Magidor, {\cite[Theorem 1.2]{MR2811288}}]\label{thm:EffectsOfMM}
Assume \MM holds, and let $\lambda$ be an uncountable cardinal.
\begin{enumerate}
  \item                                 If $\cf(\lambda)=\omega$, then $\square^*_\lambda$ fails.
  \item \label{item:CofOmega1}          If $\cf(\lambda)=\omega_1$, then $\square_{\lambda,\mu}$ fails for every $\mu<\lambda$.
  \item \label{item:Cof>=Omega2}        If $\cf(\lambda)\ge\omega_2$, then $\square_{\lambda,\mu}$ fails for every $\mu<\cf(\lambda)$.
\end{enumerate}
\end{fact}

Note that, in addition, \MM implies the failure of $\square^*_{\omega_1}$.
To see this, recall that by a theorem of Baumgartner (see \cite[Thm.~7.7]{MR776625}), \PFA (and hence \MM) implies that there is no $\omega_2$-Aronszajn tree,
while $\square^*_{\omega_1}$ is equivalent to the existence of a special $\omega_2$-Aronszajn tree.

To show that their result is optimal, Cummings and Magidor also proved the following.

\begin{fact}[Cummings and Magidor, {\cite[Theorem 1.3]{MR2811288}}]\label{fact32}
It is consistent that $\MM$ holds and for all cardinals $\lambda>\omega_1$:\footnote{There is a slight inaccuracy in the original formulation of \cite[Theorem 1.3]{MR2811288}, where it is not required that $\lambda>\omega_1$, and as a result it is claimed that $\MM$ is consistent with $\square^*_{\omega_1}$, which is not true, as pointed out above.}
\begin{enumerate}
  \item If $\cf(\lambda)=\omega_1$,     then $\square^*_\lambda$ holds.
  \item If $\cf(\lambda)\ge\omega_2$,   then $\square_{\lambda,\cf(\lambda)}$ holds.
\end{enumerate}
\end{fact}

The main motivation for investigating the forcing axiom for the class of subcomplete forcing is that on the face of it, this class is very different from the other forcing classes considered, because subcomplete forcing does not add reals. Jensen \cite{Jensen:FAandCH}, \cite{MR2840749}
proved the consistency of \SCFA, starting from a model of \ZFC with a supercompact cardinal, by adapting the Baumgartner proof of the consistency of \PFA from the same assumption, exploiting the fact, proven in \cite{Jensen:SPSCF}, \cite{MR2840749},
that subcomplete forcing is iterable with revised countable support. Since subcomplete forcing does not add reals, it follows that \SCFA is compatible with the continuum hypothesis, while the other forcing axioms imply the failure of \ch. In fact, \SCFA is even consistent with $\diamondsuit$. However, \SCFA does not imply \ch, since \MM implies $\SCFA + \neg\ch$. It is maybe a little surprising, then, that \SCFA turns out to have many of the same consequences \MM has. Thus, Jensen showed \cite{Jensen:FAandCH}, \cite{MR2840749}
that \SCFA implies the singular cardinal hypothesis and the failure of $\square_\kappa$ for every uncountable cardinal $\kappa$.

It was observed by the first author in \cite[Theorem 2.7 and Observation 2.8]{Fuchs:HierarchiesOfForcingAxioms} that it follows from Jensen's work that \SCFA implies the principle $\refl(\omega_1,E^\mu_\omega)$
for every regular cardinal $\mu>\omega_1$, and hence, by Fact~\ref{squareNreflection},
that Clauses (\ref{item:CofOmega1}) and (\ref{item:Cof>=Omega2}) of Fact~\ref{thm:EffectsOfMM} are also consequences of \SCFA, see \cite[Theorem 2.11]{Fuchs:HierarchiesOfForcingAxioms}. However, it was left open whether \SCFA implies the failure of $\square^*_\lambda$ when $\lambda$ is a singular cardinal of countable cofinality, see \cite[Question 2.12]{Fuchs:HierarchiesOfForcingAxioms}.

In order to answer this question, we will establish a version for subcomplete forcing of a well-known result of Larson \cite[Theorem 4.3]{MR1782117},
asserting that Martin's Maximum is preserved by ${<}\omega_2$-directed closed forcing.
The point is that Fact~\ref{thm:EffectsOfMM} admits the following abstract generalization:

\begin{thm}\label{thm33} Assume $2^{\aleph_0}<\aleph_\omega$ and that for every ${<}\aleph_{\omega+1}$-directed closed notion of forcing $\mathbb Q$,
we have $\V^{\mathbb Q}\models\refl(\omega_1,E^\mu_\omega)$ for every regular cardinal $\mu>\omega_1$.

Let $\lambda$ be an uncountable cardinal.
\begin{enumerate}
  \item If $\cf(\lambda)=\omega$, then $\square^*_\lambda$ fails.
  \item If $\cf(\lambda)=\omega_1$, then $\square_{\lambda,\mu}$ fails for every $\mu<\lambda$.
  \item If $\cf(\lambda)\ge\omega_2$, then $\square_{\lambda,\mu}$ fails for every $\mu<\cf(\lambda)$.
\end{enumerate}
\end{thm}
\begin{proof} (1) Fix an arbitrary singular cardinal $\lambda$ of countable cofinality.
We verify that Corollary~\ref{cor26} applies with $\kappa:=2^{\aleph_0}$, $\theta:=\aleph_0$, and $\mathbb P$ the trivial forcing.
Evidently, $\V^{\mathbb P}\models \kappa^\theta=\kappa$.
As $\lambda\ge\aleph_\omega$ and $\mathbb Q:=\mathbb P*\add(\lambda^+,1)$ is ${<}\lambda^+$-directed closed, we have $\V^{\mathbb P*\add(\lambda^+,1)}\models\refl(\omega_1,E^\mu_\omega)$ for every regular cardinal $\mu>\omega_1$.
In particular, $\V^{\mathbb P*\add(\lambda^+,1)}\models\refl(\omega,E^\mu_\omega)$ for every regular cardinal $\mu$ with $\lambda>\mu>\kappa$.

(2) By Fact~\ref{squareNreflection}(1).

(3) By Fact~\ref{squareNreflection}(2).
\end{proof}

Before we can prove that \SCFA is preserved by ${<}\omega_2$-directed closed forcing, we need some terminology.

\begin{defn} A transitive set $N$ (usually a model of $\ZFCm$) is \emph{full} if there is an ordinal $\gamma$ such that $L_\gamma(N)\models\ZFCm$ and $N$ is regular in $L_\gamma(N)$, meaning that if $x\in N$, $f\in L_\gamma(N)$ and $f:x\longrightarrow N$, then $\ran(f)\in N$.
\end{defn}

\begin{defn} For a poset $\mathbb P$, $\delta(\mathbb P)$ is the minimal cardinality of a dense subset of $\mathbb P$.
\end{defn}

\begin{defn} Let $N=L^A_\tau=\kla{L_\tau[A],\in,A\cap L_\tau[A]}$ be a $\ZFCm$ model, $\varepsilon$ an ordinal and $X\cup\{\varepsilon\}\subseteq N$. Then $C^N_\varepsilon(X)$ is the smallest $Y\prec N$ (with respect to inclusion) such that $X\cup\varepsilon\subseteq Y$.
\end{defn}

Note that models $N$ of the form described in the previous definition have definable Skolem-functions, so that the definition of $C^N_\varepsilon(X)$ makes sense. The following concept was introduced in \cite{Fuchs:ParametricSubcompleteness}.

\begin{defn} Let $\varepsilon$ be an ordinal.
A forcing $\mathbb P$ is $\varepsilon$-\emph{subcomplete} if there is a cardinal $\theta>\varepsilon$ which verifies the $\varepsilon$-subcompleteness of $\mathbb P$, which means that $\mathbb P\in H_\theta$, and for any $\ZFCm$ model $N=L_\tau^A$ with $\theta<\tau$ and $H_\theta\subseteq L_\tau[A]$, any $\sigma:\bar N\prec N$ such that $\bar N$ is countable, transitive and full and such that $\mathbb P,\theta,\varepsilon\in\ran(\sigma)$, any $\bar G\subseteq\bar{\mathbb P}$ which is $\bar{\mathbb P}$-generic over $\bar N$, and any $s\in\ran(\sigma)$, the following holds. Letting $\sigma(\kla{\bar{s},\bar{\theta},\bar{\mathbb P}})=\kla{s,\theta,\mathbb P}$, there is a condition $p\in\mathbb P$ such that whenever $G\subseteq\mathbb P$ is $\mathbb P$-generic over $\V$ with $p\in G$, there is in $\V[G]$ a $\sigma'$ such that
\begin{enumerate}
  \item $\sigma':\bar N\prec N$,
  \item $\sigma'(\kla{\bar{s},\bar{\theta},\bar{\mathbb P},\bar{\varepsilon}})=\kla{s,\theta,\mathbb P,\varepsilon}$,
  \item $(\sigma')``\bar G\subseteq G$,
  \item $C^N_{\varepsilon}(\ran(\sigma'))=C^N_{\varepsilon}(\ran(\sigma))$.
\end{enumerate}
\end{defn}

Using this terminology, subcompleteness is a special case of $\varepsilon$-subcompleteness. Essential subcompleteness was introduced in \cite{Fuchs:ParametricSubcompleteness}.

\begin{defn} A poset $\mathbb P$ is \emph{subcomplete} iff $\mathbb P$ is $\delta(\mathbb P)$-subcomplete. $\mathbb P$ is \emph{essentially subcomplete} iff $\mathbb P$ is $\varepsilon$-subcomplete, for some $\varepsilon$.
\end{defn}

Increasing $\varepsilon$ weakens the condition of being $\varepsilon$-subcomplete, but if a forcing $\mathbb P$ is $\varepsilon$-subcomplete, for an $\varepsilon>\delta(\mathbb P)$, then $\mathbb P$ is forcing equivalent to a subcomplete forcing, where we say that two forcing notions are forcing equivalent if they give rise to the same forcing extensions. In fact, it is shown in \cite{Fuchs:ParametricSubcompleteness} that the class of essentially subcomplete forcing notions is the closure of the class of subcomplete forcing notions under forcing equivalence.

\begin{defn} Let $\mathbb P$ be a notion of forcing, and $p\in\mathbb P$ a condition. Then $\mathbb P_{{\le}p}$ is the restriction of $\mathbb P$ to $\{q\mid q\le_{\mathbb P} p\}$.
\end{defn}

\begin{lemma} Let $\mathbb P$ be $\varepsilon$-subcomplete and $p\in\mathbb P$. Then $\mathbb P_{{\le}p}$ is $\varepsilon$-subcomplete.
\end{lemma}
\begin{proof}
Let $\theta$ verify the $\varepsilon$-subcompleteness of $\mathbb P$. To show that $\mathbb P_{{\le}p}$ is $\varepsilon$-subcomplete, let $N=L_\tau^A$, $\theta<\tau$, $\bar N$ countable and full and $\sigma:\bar N\prec N$ with $\mathbb P_{{\le}p},\theta,\varepsilon\in\ran(\sigma)$.
By the argument of \cite[p.~116]{MR2840749}, we may also assume that $\mathbb P,p\in\ran(\sigma)$. Let $s\in\ran(\sigma)$ be given. Further, let $\bar{s}=\sigma^{-1}(s)$, $\bar{\theta}=\sigma^{-1}(\theta)$, $\bar{\varepsilon}=\sigma^{-1}(\varepsilon)$, $\bar{\mathbb P}=\sigma^{-1}(\mathbb P)$ and $\bar{p}=\sigma^{-1}(p)$, so that $\sigma(\bar{\mathbb P}_{{\le}\bar{p}})=\mathbb P_{{\le}p}$. Finally, let $\bar G$ be $\bar{\mathbb P}_{{\le}\bar{p}}$-generic over $\bar N$. Then $\bar G$ generates a filter $\bar{H}=\{q\in\bar{\mathbb P}\mid\exists r\in\bar G\quad r\le q\}$ that is $\bar{\mathbb P}$-generic over $\bar N$. Since $\theta$ verifies the $\varepsilon$-subcompleteness of $\mathbb P$, there is a condition $r\in\mathbb P$ such that whenever $H$ is $\mathbb P$-generic over $\V$, then in $\V[H]$, there is a $\sigma'$ such that
\begin{enumerate}
  \item $\sigma':\bar N\prec N$,
  \item $\sigma'(\bar{s},\bar{\theta},\bar{\mathbb P},\bar{p})=\kla{s,\theta,\mathbb P,p}$,
  \item $(\sigma')``\bar{H}\subseteq H$,
  \item $C^N_\varepsilon(\ran(\sigma'))=C^N_\varepsilon(\ran(\sigma))$.
\end{enumerate}
Note that it follows from (2) and (3) that $p\in H$. In particular, $r$ is compatible with $p$. Letting $r'\le p,r$, it follows that $r'\in\mathbb P_{{\le}p}$, and whenever $G$ is $\mathbb P_{{\le}p}$-generic over $\V$ with $r'\in G$, then there is a $\sigma'$ in $\V[G]$ such that
\begin{enumerate}
  \item[(1')] $\sigma':\bar N\prec N$,
  \item[(2')] $\sigma'(\bar{s},\bar{\theta},\bar{\mathbb P}_{{\le}\bar{p}})=\kla{s,\theta,\mathbb P_{{\le}p}}$,
  \item[(3')] $(\sigma')``\bar G\subseteq G$,
  \item[(4')] $C^N_\varepsilon(\ran(\sigma'))=C^N_\varepsilon(\ran(\sigma))$.
\end{enumerate}
This is because if $G$ is as described, then it generates an $H$ which is $\mathbb P$-generic, and since $r'\in G$, it follows that $r\in H$. Thus, there is a $\sigma'\in\V[H]$ that satisfies (1)-(4). But $\V[H]=\V[G]$, and it follows that $\sigma'$ satisfies (1')-(4').
\end{proof}

\begin{cor}\label{cor:RestrictionsOfSCforcingAreDelta(P)SC} If $\mathbb P$ is subcomplete and $p\in\mathbb P$, then $\mathbb P_{{\le}p}$ is $\delta(\mathbb P)$-subcomplete.
\end{cor}

Note that in the situation of this corollary, since $\delta(\mathbb P_{{\le}p})$ may be smaller than $\delta(\mathbb P)$, we do not necessarily know that $\mathbb P_{{\le}p}$ is subcomplete. We are now ready to prove the requisite preservation property of \SCFA.

\begin{lemma}\label{lem:SCFApreservedByDirectedClosedForcing} \SCFA is preserved by ${<}\omega_2$-directed closed forcing.
\end{lemma}
\begin{proof}
The basic idea of the proof of \cite[Theorem 4.3]{MR1782117} goes through, but matters are slightly complicated by dealing with subcomplete forcing.

Let $\mathbb P$ be ${<}\omega_2$-directed closed, and assume that \SCFA holds. Let $G$ be $\mathbb P$-generic over $\V$, and in $\V[G]$, let $\mathbb Q$ be a subcomplete forcing, and let $\vec{\Delta}=\seq{\Delta_\xi}{\xi<\omega_1}$ be a sequence of dense subsets of $\mathbb Q$. We have to show that there is a $\vec{\Delta}$-generic filter for $\mathbb Q$ in $\V[G]$.

Assume, towards a contradiction, that there is no filter in $\V[G]$ which meets the sets in the sequence $\vec{\Delta}$, and let $\dot{\Delta}$ be a $\mathbb P$-name for $\vec{\Delta}$. Let $p\in\mathbb P$ be a condition which forces that $\vec{\Delta}$ has the properties described above, and that there is no $\dot{\Delta}$-generic filter. By Corollary~\ref{cor:RestrictionsOfSCforcingAreDelta(P)SC},
$\mathbb P_{{\le}p}$ is $\delta(\mathbb P)$-subcomplete, and in particular, it is essentially subcomplete.

If $H\subseteq\mathbb P_{{\le}p}$ is generic for $\mathbb P_{{\le}p}$, then let us write $H'$ for the filter over $\mathbb P$ generated by $H$, that is, $H'=\{r\in\mathbb P\mid\exists q\in H\quad r\ge q\}$. It follows that $H'$ is $\mathbb P$-generic, because if $D\subseteq\mathbb P$ is dense, then $D\cap\mathbb P_{{\le}p}$ is dense in $\mathbb P_{{\le}p}$. Moreover, if $I\subseteq\mathbb P$ is a $\mathbb P$-generic filter containing $p$, then $H:=I\cap\mathbb P_{{\le}p}$ is a generic filter over $\mathbb P_{{\le}p}$, and $I=H'$. There is a simple recursive procedure to translate any $\mathbb P$-name $\tau$ to a $\mathbb P_{{\le}p}$-name $\tau'$ in such a way that whenever $I$ and $H$ are as just described, it follows that $\tau^I=(\tau')^H$. This can be achieved by the recursive definition \[\tau'=\{\kla{\sigma',q}\mid q\le p\ \land\ \exists r\ge q\quad\kla{\sigma,r}\in\tau\}.\]
It is then clear that for every formula $\varphi(\tau_0,\ldots,\tau_{n-1})$ of the forcing language for $\mathbb P$, and for every $q\le p$, we have that
\[q\Vdash_{\mathbb P}\varphi(\tau_0,\ldots,\tau_{n-1})\iff
q\Vdash_{\mathbb P_{{\le}p}}\phi(\tau'_0,\ldots,\tau'_{n-1}).\]
To see this, from left to right, assume that $H$ is a $\mathbb P_{{\le}p}$-generic filter with $q\in H$. We know that then, $I=H'$ is $\mathbb P$-generic and $q\in H'$, and so, by assumption, in $\V[H']$, $\phi(\tau_0^{H'},\ldots,\tau_{n-1}^{H'})$ holds. Moreover, $H=I\cap\mathbb P_{{\le}p}$, and so, we know that $\V[H']=\V[H]$ and $\tau_i^{H'}=(\tau'_i)^H$, for $i<n$, and hence, $\V[H]\models\phi((\tau'_0)^H,\ldots,(\tau'_{n-1})^H)$, which means that $q\Vdash_{\mathbb P_{{\le}p}}\phi(\tau'_0,\ldots,\tau'_{n-1})$. The converse follows similarly.

Thus, we can replace the $\mathbb P$-names $\dot{\Delta}$ and $\dot{\mathbb Q}$ with the translated $\mathbb P_{{\le}p}$-names $\dot{\Delta}'$ and $\dot{\mathbb Q}'$, and we then know that $p=1_{\mathbb P_{{\le}p}}$ forces with respect to $\mathbb P_{{\le}p}$ that $\dot{\mathbb Q}'$ is a subcomplete forcing notion, $\dot{\Delta}'$ is an $\omega_1$-sequence of dense subsets of $\dot{\mathbb Q}'$ and there is no $\dot{\Delta}'$-generic filter.

For $\zeta<\omega_1$, define a dense subset $D_\zeta\subseteq\mathbb P_{{\le}p}*\dot{\mathbb Q}'$ by setting
\[D_\zeta=\{\kla{q,\sigma}\in\mathbb P_{{\le}p}*\dot{\mathbb Q}'\mid p\Vdash_{\mathbb P_{{\le}p}}\sigma\in\dot{\Delta}'_{\check{\zeta}}\}.\]
Since $\mathbb P_{{\le}p}$ is essentially subcomplete and $\mathbb P_{{\le}p}$ forces that $\dot{\mathbb Q}'$ is (essentially) subcomplete, it follows that $\mathbb P_{{\le}p}*\dot{\mathbb Q}'$ is essentially subcomplete (by \cite[Theorem 2.9]{Fuchs:ParametricSubcompleteness}). By \cite[Lemma 2.6]{Fuchs:ParametricSubcompleteness}, \SCFA implies the forcing axiom for essentially subcomplete forcing notions, and so, there is in $\V$ a filter $F$ in $\mathbb P_{{\le}p}*\dot{\mathbb Q}'$ that meets $D_\zeta$, for every $\zeta<\omega_1$. It is now straightforward to construct a subset $\bar F\subseteq F$ such that
\begin{enumerate}
    \item                                   the cardinality of $\bar F$ is at most $\omega_1$,
    \item                                   for every $\zeta<\omega_1$, $\bar F\cap D_\zeta\neq\emptyset$,
    \item \label{item:ClosedUnderWitnesses} for all $s,t\in\bar F$, there is a $u\in\bar F$ such that $u\le s,t$.
\end{enumerate}
The collection of the first coordinates of conditions in $\bar F$ is a directed subset of $\mathbb P_{{\le}p}$ of size at most $\omega_1$, so by ${<}\omega_2$-directedness, we may choose a condition $q$ such that for every $\kla{r,\sigma}\in\bar F$, we have that $q\le r$. Letting $\tilde{F}=\{\kla{\sigma,r}\mid\kla{r,\sigma}\in F\}$, we claim now that $q$ forces with respect to $\mathbb P_{{\le}p}$ that $\tilde{F}$ generates a $\dot{\Delta}'$-generic filter over $\dot{\mathbb Q}'$.

To see this, let $H$ be $\mathbb P_{{\le}p}$-generic with $q\in H$. First, let's check that $\tilde{F}^H$ generates filter. Let $a,b\in\tilde{F}^H$. It suffices to show that there is a $c\in\tilde{F}^H$ with $c\le a,b$. Let $\kla{\sigma_1,r_1},\kla{\sigma_2,r_2}\in\tilde{F}$ be such that $a=\sigma_1^H$, $b=\sigma_2^H$ and $r_1,r_2\in H$. By (\ref{item:ClosedUnderWitnesses}), there is a condition $\kla{r_3,\sigma_3}\in\bar F$ with $\kla{r_3,\sigma_3}\le\kla{r_1,\sigma_1},\kla{r_2,\sigma_2}$. This means that $r_3\le r_1,r_2$ and $r_3$ forces that $\sigma_3\le\sigma_1,\sigma_2$. Since we made sure that $q\le r_1,r_2,r_3$ and $q\in H$, it follows that $c=\sigma_3^H\le b,c$ and $c\in\tilde{F}^H$. Finally, if $\xi<\omega_1$, then since $\bar F\cap D_\xi\neq\emptyset$, there is a $\kla{r,\sigma}\in\bar F$ with $r\Vdash\sigma\in\dot{\Delta}_{\check{\xi}}$, so that, since $\kla{\sigma,r}\in\tilde{F}$, $q\le r$ and $q\in H$, we get that $\sigma^H\in\tilde{F}^H\cap(\dot{\Delta}')^H_\xi$.

This contradicts our assumption that $p$ forces that there is no $\dot{\Delta}'$-generic filter over $\dot{\mathbb Q}'$.
\end{proof}

As a result, we arrive at the following complete description of the effects of \SCFA on weak square principles, in particular,
answering Question~2.12 of \cite{Fuchs:HierarchiesOfForcingAxioms}.

\begin{cor}\label{cor313} Assume \SCFA holds, and let $\lambda$ be an uncountable cardinal.
\begin{enumerate}
  \item If $\cf(\lambda)=\omega$, then $\square^*_\lambda$ fails.
  \item If $\cf(\lambda)=\omega_1$, then $\square_{\lambda,\mu}$ fails for every $\mu<\lambda$.
  \item If $\cf(\lambda)\ge\omega_2$, then $\square_{\lambda,\mu}$ fails for every $\mu<\cf(\lambda)$.
\end{enumerate}
\end{cor}
\begin{proof} By \cite{Jensen:FAandCH}, \cite{MR2840749}, \SCFA implies that $2^{\aleph_1}=\aleph_2$,\footnote{In fact, one can show that \SCFA entails $\diamondsuit(E^{\omega_2}_{\omega_1})$.} so that $2^{\aleph_0}<\aleph_\omega$.
As pointed out earlier, \SCFA implies $\refl(\omega_1,E^\mu_\omega)$ for every regular cardinal $\mu>\omega_1$.
In particular, by Lemma~\ref{lem:SCFApreservedByDirectedClosedForcing},
for every ${<}\aleph_{\omega+1}$-directed closed notion of forcing $\mathbb Q$,
$\V^{\mathbb Q}\models\refl(\omega_1,E^\mu_\omega)$ for every regular cardinal $\mu>\omega_1$.
Now, appeal to Theorem~\ref{thm33}.
\end{proof}

Thus, we have proved Corollary~B.
Note that Corollary~\ref{cor313} is optimal in the following sense (see the discussion before Question 2.12 in \cite{Fuchs:HierarchiesOfForcingAxioms}).

\begin{fact}[\cite{Fuchs:HierarchiesOfForcingAxioms}] It is consistent that \SCFA holds and for every uncountable cardinal $\lambda$, the following hold:
\begin{enumerate}
  \item \label{item:CofOmega_1} If $\cf(\lambda)=\omega_1$, then $\square^*_\lambda$ holds.
  \item                         If $\cf(\lambda)\ge\omega_2$, then $\square_{\lambda,\cf(\lambda)}$ holds.
\end{enumerate}
\end{fact}
\begin{remark}
The preceding is witnessed by a model of $\SCFA+\ch$.
Thus, unlike Fact~\ref{fact32}, Clause~(\ref{item:CofOmega_1}) of the preceding does apply to $\lambda=\omega_1$ (since $\ch\implies\square^*_{\omega_1}$.)
\end{remark}


\begin{thebibliography}{CFM01}
\bibitem[CFM01]{MR1838355}
James Cummings, Matthew Foreman, and Menachem Magidor.
\newblock Squares, scales and stationary reflection.
\newblock {\em J. Math. Log.}, 1(1):35--98, 2001.

\bibitem[CM11]{MR2811288}
James Cummings and Menachem Magidor.
\newblock Martin's maximum and weak square.
\newblock {\em Proc. Amer. Math. Soc.}, 139(9):3339--3348, 2011.

\bibitem[Dev79]{MR523488}
Keith~J. Devlin.
\newblock Variations on {$\diamondsuit $}.
\newblock {\em J. Symbolic Logic}, 44(1):51--58, 1979.

\bibitem[Fuc16]{Fuchs:HierarchiesOfForcingAxioms}
Gunter Fuchs.
\newblock Hierarchies of forcing axioms, the continuum hypothesis and square principles.
\newblock {\em J. Symbolic Logic}, submitted in 2016.
\newblock To appear. Preprint available at
  {http://www.math.csi.cuny.edu/\textasciitilde fuchs/}.

\bibitem[Fuc17]{Fuchs:ParametricSubcompleteness}
Gunter Fuchs.
\newblock Closure properties of parametric subcompleteness.
\newblock Submitted in 2017.
\newblock Preprint available at {http://www.math.csi.cuny.edu/\textasciitilde
  fuchs/}.

\bibitem[Jen09a]{Jensen:FAandCH}
Ronald~B. Jensen.
\newblock Forcing axioms compatible with {CH}.
\newblock Handwritten notes, available at
  https://www.mathematik.hu-berlin.de/\textasciitilde raesch/org/jensen.html,
  2009.

\bibitem[Jen09b]{Jensen:SPSCF}
Ronald~B. Jensen.
\newblock Subproper and subcomplete forcing.
\newblock 2009.
\newblock Handwritten notes, available at
  https://www.mathematik.hu-berlin.de/\textasciitilde raesch/org/jensen.html.

\bibitem[Jen14]{MR2840749}
Ronald Jensen.
\newblock Subcomplete forcing and {$\mathcal{L}$}-forcing.
\newblock In {\em {$E$}-recursion, forcing and {$C^*$}-algebras}, volume~27 of
  {\em Lect. Notes Ser. Inst. Math. Sci. Natl. Univ. Singap.}, pages 83--182.
  World Sci. Publ., Hackensack, NJ, 2014.

\bibitem[Kun83]{MR756630}
Kenneth Kunen.
\newblock {\em Set theory}, volume 102 of {\em Studies in Logic and the Foundations of Mathematics}.
\newblock North-Holland Publishing Co., Amsterdam, 1983.
\newblock An introduction to independence proofs, Reprint of the 1980 original.

\bibitem[Lar00]{MR1782117}
Paul Larson.
\newblock Separating stationary reflection principles.
\newblock {\em J. Symbolic Logic}, 65(1):247--258, 2000.

\bibitem[Rin10]{MR2723781}
Assaf Rinot.
\newblock A relative of the approachability ideal, diamond and non-saturation.
\newblock {\em J. Symbolic Logic}, 75(3):1035--1065, 2010.

\bibitem[Sil71]{MR0277379}
Jack Silver.
\newblock The independence of {K}urepa's conjecture and two-cardinal conjectures in model theory.
\newblock In {\em Axiomatic {S}et {T}heory ({P}roc. {S}ympos. {P}ure {M}ath.,
  {V}ol. {XIII}, {P}art {I}, {U}niv. {C}alifornia, {L}os {A}ngeles, {C}alif.,
  1967)}, pages 383--390. Amer. Math. Soc., Providence, R.I., 1971.

\bibitem[Tod84]{MR776625}
Stevo Todor\v{c}evi\'{c}.
\newblock Trees and linearly ordered sets.
\newblock In {\em Handbook of set-theoretic topology}, pages 235--293.
  North-Holland, Amsterdam, 1984.
\end{thebibliography}
\end{document}